\pdfoutput=1
\RequirePackage{ifpdf}
\ifpdf 
\documentclass[pdftex]{sigma}
\else
\documentclass{sigma}
\fi

\usepackage{xy}

\newcommand{\mff}{{\mathbb{F}}}
\newcommand{\mpp}{{\mathbb{P}}}
\newcommand{\mqq}{{\mathbb{Q}}}
\newcommand{\mzz}{{\mathbb{Z}}}

\newcommand{\mseta}[1]{{\left\langle{#1}\right\rangle}}		
\newcommand{\msetb}[1]{{\left\{{#1}\right\}}}				
\newcommand{\msetp}[1]{{\left({#1}\right)}}					
\newcommand{\msets}[1]{{\left[{#1}\right]}}					
\newcommand{\msetss}[1]{{\left[\left[{#1}\right]\right]}}	
\newcommand{\msetv}[1]{{\left|{#1}\right|}}					

\newcommand{\maut}[1]{\operatorname{Aut}\msetp{#1}}			
\newcommand{\mac}[1]{{\overline{#1}}}
\newcommand{\mchar}[1]{\operatorname{char}\msetp{#1}}			
\newcommand{\mdiv}[1]{\operatorname{div}\msetp{#1}}			
\newcommand{\mhess}[1]{\operatorname{Hess}\msetp{#1}}			
\newcommand{\mgl}[2]{{\operatorname{GL}_{#1}}\msetp{#2}}		
\newcommand{\miso}[1]{\operatorname{Iso}\msetp{#1}}			
\newcommand{\mpgl}[2]{{\operatorname{PGL}_{#1}}\msetp{#2}}	
\newcommand{\mps}[2]{{\mpp^{#1}\msetp{#2}}}						
\newcommand{\msl}[2]{{\operatorname{SL}_{#1}}\msetp{#2}}		
\newcommand{\mres}[2]{{\left.{#1}\right|_{#2}}}					

\xyoption{arrow}
\xyoption{matrix}
\xyoption{tips}

\SelectTips{cm}{11}

\numberwithin{equation}{section}
\newtheorem{Theorem}{Theorem}[section]
\newtheorem{Corollary}[Theorem]{Corollary}
\newtheorem{Lemma}[Theorem]{Lemma}
\newtheorem{Proposition}[Theorem]{Proposition}

\begin{document}

\newcommand{\arXivNumber}{1708.09745}

\renewcommand{\thefootnote}{}

\renewcommand{\PaperNumber}{102}

\FirstPageHeading

\ShortArticleName{Hesse Pencils and 3-Torsion Structures}

\ArticleName{Hesse Pencils and 3-Torsion Structures\footnote{This paper is a~contribution to the Special Issue on Modular Forms and String Theory in honor of Noriko Yui. The full collection is available at \href{http://www.emis.de/journals/SIGMA/modular-forms.html}{http://www.emis.de/journals/SIGMA/modular-forms.html}}}

\Author{Ane S.I.~ANEMA, Jaap TOP and Anne TUIJP}

\AuthorNameForHeading{A.S.I.~Anema, J.~Top and A.~Tuijp}

\Address{Bernoulli Institute for Mathematics, Computer Science and Artificial Intelligence,\\
University of Groningen, P.O.~Box~407, 9700 AK Groningen, The Netherlands}
\Email{\href{mailto:a.s.i.anema@22gd7.nl}{a.s.i.anema@22gd7.nl}, \href{mailto:j.top@rug.nl}{j.top@rug.nl}, \href{mailto:annetuijp@gmail.com}{annetuijp@gmail.com}}

\ArticleDates{Received May 08, 2018, in final form September 18, 2018; Published online September 21, 2018}

\Abstract{This paper intends to focus on the universal property of this Hesse pencil and of its twists. The main goal is to do this as explicit and elementary as possible, and moreover to do it in such a way that it works in every characteristic different from three.}


\Classification{14D10; 14G99}

\renewcommand{\thefootnote}{\arabic{footnote}}
\setcounter{footnote}{0}

\section{Introduction}

In a paper with Noriko Yui \cite{TY2007}, explicit equations for all elliptic modular surfaces corresponding to genus zero torsion-free subgroups of $\mbox{PSL}_2({\mathbb Z})$ were presented. Arguably the most famous and classical one of these surfaces is the Hesse pencil, usually described as the family of plane cubics
\begin{gather*}
x^3+y^3+z^3+6txyz=0
\end{gather*}
(see, e.g., \cite[Table~2]{TY2007} and references given there). The current paper intends to focus on the universal property of this Hesse pencil and of its twists (see Theorem~\ref{hpga:thm:mt}). The main goal is to do this as explicit and elementary as possible, and moreover to do it in such a way that it works in every characteristic different from three. We are not aware of any earlier publication of these results in the special case of characteristic two, so although admittedly not very difficult, those appear to be new. The first author of this paper worked on the present results as a (small) part of his Ph.D.~Thesis \cite{ane} and the third author did the same as part of her Bachelor's Thesis~\cite{tuijp2015}. Both were supervised by the second author.

Let $k$ be a perfect field of characteristic different from three. Denote the absolute Galois group of $k$ by $G_k$. Given an elliptic curve~$E$ defined over $k$, one obtains a Galois representation on the 3-torsion group $E \msets{3}$ of $E$. This paper describes the family of all elliptic curves that have equivalent Galois representations on $E\msets{3}$. Recall that elliptic curves $E$ and $E^\prime$ over $k$ yield equivalent Galois representations on their 3-torsion if and only if $E\msets{3}$ and $E^\prime\msets{3}$ are isomorphic as $G_k$-modules. To be more specific, we demand that the equivalence is symplectic: a symplectic homomorphism $\phi\colon E[3]\to E'[3]$ is defined as in~\cite{kuwata2012}, as follows. If
\begin{gather*}
e_3\msetp{S,T}=e^\prime_3\msetp{\phi\msetp{S},\phi\msetp{T}}
\end{gather*}
for all $S,T\in E\msets{3}$ where $e_3$ and $e^\prime_3$ are the Weil-pairings on the 3-torsion of $E$ and $E^\prime$ respectively, then $\phi$ is called a~\emph{symplectic} homomorphism, otherwise~$\phi$ is called an \emph{anti-symplectic} homomorphism.

Next, recall the definition of the Hessian of a polynomial. Let $F\in k\msets {X,Y,Z}$ be a homogeneous polynomial of degree $n$. The \emph{Hessian} $\mhess{F}$ of $F$ is the determinant of the Hessian matrix of $F$, that is
\begin{gather*}
\mhess{F}=\det
\begin{pmatrix}
\dfrac{\partial^2 F}{\partial X^2} & \dfrac{\partial^2 F}{\partial X \partial Y} & \dfrac{\partial^2 F}{\partial X \partial Z} \vspace{1mm}\\
\dfrac{\partial^2 F}{\partial X \partial Y} & \dfrac{\partial^2 F}{\partial Y^2} & \dfrac{\partial^2 F}{\partial Y \partial Z} \vspace{1mm}\\
\dfrac{\partial^2 F}{\partial X \partial Z} & \dfrac{\partial^2 F}{\partial Y \partial Z} & \dfrac{\partial^2 F}{\partial Z^2}
\end{pmatrix},
\end{gather*}
which is either a homogeneous polynomial of degree $3n-6$ or zero.

Given a curve $C=Z\msetp{F}$ with $F\in k\msets{X,Y,Z}$ homogeneous of degree three, the \emph{Hesse pencil} of $C$ is defined as
\begin{gather*}
\mathcal{C}=Z\msetp{tF+\mhess{F}}
\end{gather*}
over $k\msetp{t}$. Recall that the discrete valuations on $k\msetp{t}$ correspond to the points in $\mps{1}{k}$, where we usually write $\msetp{t_0:1}$ as $t_0$ and $\msetp{1:0}$ as $\infty$. We denote the reduced curve of $\mathcal {C}$ at $t_0\in\mps{1}{k}$ by $C_{t_0}$. Notice that $C_\infty=C$ and for $t_0 \neq\infty$
\begin{gather*}
C_{t_0}=Z\msetp{t_0F+\mhess{F}}.
\end{gather*}

In the special case that $C=E$ is an elliptic curve given by a~Weierstrass equation, we have (see Section~\ref{hpga:sec:fp}) that the point $O$ at infinity is a~point on $E_{t_0}$ for every $t_0\in\mps{1}{k}$. If $E_{t_0}$ is a~smooth curve, then this makes it an elliptic curve with unit element $O$.

In the case of characteristic two, the standard definition of the Hessian does not lead to a~satisfactory theory. In Sections~\ref{8.2} and~\ref{8.3} a modified Hessian is introduced for this case; in fact this modification was already used by Dickson~\cite{Dickson} in 1915. The goal of this paper is to provide an elementary proof of the following theorem:

\begin{Theorem}\label{hpga:thm:mt} If $E$ and $E^\prime$ are elliptic curves over $k$, with $E$ given by some Weierstrass equation over $k$, then there exists a symplectic isomorphism $E\msets{3}\rightarrow E^\prime\msets{3}$ if and only if $E'$ appears in the Hesse pencil of $E$, i.e., $E_{t_0}\cong_kE^\prime$ for some $t_0\in\mps{1}{k}$.
\end{Theorem}

We note that both Fisher's paper \cite{fisher2012} and Kuwata's paper \cite{kuwata2012} discuss, apart from the result above (although not in characteristic two) also the case of anti-symplectic isomorphisms between the $3$-torsion groups of elliptic curves.

In Sections~\ref{hpga:sec:fp}, \ref{hpga:sec:3tgrp} and~\ref{hpga:sec:wp} we show that the 3-torsion groups of an elliptic curve in Weierstrass form and its Hesse pencil are identical not only as sets, but also have the same group structure and Weil-pairings. Using the Weierstrass form of the Hesse pencil computed in Section~\ref{hpga:sec:wf} and the relation between a linear change of coordinates and its restriction to the 3-torsion group described in Section~\ref{hpga:sec:3tlcc}, we prove in Section~\ref{hpga:sec:wp3tiie} essentially by a counting argument that an
isomorphism of the 3-torsion groups respecting the Weil-pairings is the restriction of a linear change of coordinates. The proof of the theorem is completed in Section~\ref{hpga:sec:pot}. After this, we adapt the argument in order to conclude the same result in characteristic $2$ (where a slightly adapted notion of Hesse pencil is required). We compare our results with existing literature in Section~\ref{hpga:sec:cwer}.

\section{The flex points}\label{hpga:sec:fp}

Let $C=Z\msetp{F}$ be a plane curve with $F\in k\msets{X,Y,Z}$ homogeneous of degree $n$ and irreducible. A point~$P$ on~$C$ is called a \emph{flex point} if there exists a line~$L$ such that the intersection number of~$C$ and $L$ at $P$ is at least three. Notice that in our definition $P$ is allowed to be a singular point on $C$.

The \emph{Hessian curve} of $C$ is defined as $\mhess{C}=Z\msetp{\mhess{F}}$.

\begin{Proposition}\label{two} If $P$ is a point on $C$ and $\mchar{k}\nmid (n-1)$, then $P$ is a flex point if
and only if $P\in C\cap\mhess{C}$.
\end{Proposition}
\begin{proof}See \cite[Exercise 5.23]{fulton2008}.
\end{proof}

From now on we will only work with curves of degree three, so the proposition above is only usable for fields $k$ of characteristic different from two. This is the reason for why we exclude characteristic two for now; see Section~\ref{hpha:sec:char2} for the excluded case.

\begin{Corollary}\label{corthree} If $P$ is a flex point on $C$, then it is also a point on the Hesse pencil $\mathcal{C}$ and it is again a flex point of each curve of the Hesse pencil.
\end{Corollary}

This is a well-known and old result in the case of $F=X^3+Y^3+Z^3$, see for example \cite[Section~VII.1]{pascal1902}.

\begin{proof}A computation using Magma~\cite{bosma1997} shows that
\begin{gather*}
\mhess{tF+\mhess{F}} =\alpha F+\beta \mhess{F}
\end{gather*}
with $\alpha,\beta\in k\msets{t}$.

Assume that $P$ is a flex point, then $P\in C\cap\mhess{C}$ by Proposition~\ref{two}, that is $F\msetp{P}=0$ and $\mhess{F}\msetp{P}=0$. So $\msetp{tF+\mhess{F}}\msetp{P}=0$, which implies that $P\in\mathcal{C}$. The computation above also implies that
\begin{gather*}
\mhess{tF+\mhess{F}}\msetp{P}=0,
\end{gather*}
that is $P\in\mhess{\mathcal{C}}$. Therefore $P\in\mathcal{C}\cap\mhess{\mathcal{C}}$. Hence $P$ is a flex point on $\mathcal{C}$ by Proposi\-tion~\ref{two}.
\end{proof}

\begin{Corollary}\label{hpga:cor:cphpfp}
Let $P\in C_{t_0}\cap C_{t_1}$. If $t_0\neq t_1$, then $P$ is a flex point on~$C$.
\end{Corollary}
\begin{proof}
Suppose that $t_0=\msetp{t_{00}:t_{01}}$ and $t_1=\msetp{t_{10}:t_{11}}$, then
\begin{gather*}
\begin{pmatrix}
t_{00} & t_{01} \\
t_{10} & t_{11}
\end{pmatrix}
\begin{pmatrix}
F\msetp{P} \\
\mhess{F}\msetp{P}
\end{pmatrix}
=
\begin{pmatrix}
0 \\
0
\end{pmatrix},
\end{gather*}
with the matrix being invertible since $t_0\neq t_1$. Thus $F\msetp{P}=0$ and $\mhess{F}\msetp{P}=0$, that is $P\in C\cap\mhess{C}$. Hence Proposition~\ref{two} implies that $P$ is a flex point on $C$.
\end{proof}

\section{The 3-torsion group}\label{hpga:sec:3tgrp}

Let $E=Z\msetp{F}$ be an elliptic curve with unit element $O$ and $F\in k\msets {X,Y,Z}$ homogeneous of degree 3. Recall the following well known fact.

\begin{Proposition}Let $S$ and $T$ be points on $E$. If $S$ is a flex point, then $T$ is a flex point if and only if $S-T\in E\msets{3}$.
\end{Proposition}
\begin{proof}Let $L_S$ and $L_T$ be the tangent lines to $E$ at $S$ and $T$ respectively. Assume that $T$ is also a flex point. Consider the function $\frac{L_S}{L_T}$ on $E$ which has divisor $3\msetp{S}-3\msetp{T}$. From \cite[Corollary~III.3.5]{silverman1985} it follows that $3S-3T=O$. Hence $S-T\in E\msets{3}$.

Assume that $T$ is not a flex point. Now the divisor of the function $\frac {L_S}{L_T}$ is $3\msetp{S}-2\msetp{T}-\msetp{T^\prime}$ with $T^\prime\neq T$. From this it follows that $3S-2T-T^\prime=O$, thus $3S-3T=T^\prime-T\neq O$. Hence $S-T\notin E\msets{3}$.
\end{proof}

This result tells us that if $O$ is a flex point on $E$, then the concepts of flex point and 3-torsion point coincide. In the previous section we learned that a flex point on $E$ is also a flex point on $\mathcal{E}=\mathcal{Z}(tF+\mhess{F})$. Hence if we combine these statements, then we obtain $E\msets{3}\subset\mathcal{E}\msets{3}$. Since the characteristic of $k$ is different from three, these sets are equal in size, thus the same. Moreover suppose that $E_{t_0}$ for some $t_0\in\mps{1}{k}$ is non-singular. Provide $\mathcal{E}$ and $E_{t_0}$ with a group structure by taking $O$ as the unit element. Since the flex points of $\mathcal{E}$ (considered as a plane cubic over $k(t)$) and of $E_{t_0}$ (a cubic curve over $k$) are the same and a line that intersects an elliptic curve at two flex points will also intersect the curve at a third flex point, the group structures on $\mathcal{E}\msets{3}$ and $E_{t_0}\msets{3}$ are equal as well.

Recall that if the unit element $O$ is a flex point on $E$, then we can find a~projective linear transformation in $\mpgl{3}{k}$ such that $E$ is given by a~Weierstrass equation in the new coordinates. Moreover since the characteristic of $k$ is different from two and three, we may even assume that $E\colon y^2=x^3+ax+b$ for some $a,b\in k$.

\section{The Weil-pairing}\label{hpga:sec:wp}

In the previous section we saw that $\mathcal{E}\msets{3}=E_{t_0}\msets{3}$ for all $t_0\in\mps{1}{\mac{k}}$ such that $E_{t_0}$ is non-singular. Denote the Weil-pairing on the 3-torsion of $\mathcal{E}$ by $e_3$ and on the 3-torsion of $E_{t_0}$ by $e_3^{t_0}$. An introduction to Weil-pairings can be found in \cite[Section~III.8]{silverman1985} and \cite[Sections~3.3 and~11.2]{washington2008}.

\begin{Proposition}Let $E$ be an elliptic curve given by a Weierstrass equation and let $\mathcal {E}$ be its Hesse pencil. The Weil-pairings~$e_3$ and $e_3^{t_0}$ on $E\msets {3}$ are equal.
\end{Proposition}
\begin{proof}Let $S,T\in E\msets{3}$ generate $E\msets{3}$. The Weil-pairing is determined by its value on $\msetp{S,T}$. Follow \cite[Exercise~3.16]{silverman1985} to construct the Weil-pairings. Recall that $O$ is a flex point on $E$.

Let $L_O$, $L_S$, $L_T$ and $L_{-T}$ be the tangent lines to $\mathcal{E}$ at $O$, $S$, $T$ and $-T$ respectively. Define $D_S=\msetp{S}-\msetp{O}$ and $D_T= 2\msetp{T}-2\msetp{-T}$. Notice that $D_S$ and $D_T$ have disjoint support. Since $2T-2\msetp{-T}=T$ in $\mathcal{E}$, it follows that $D_T\sim\msetp{T}- \msetp{O}$. Consider the functions $f_S=\frac{L_S}{L_O}$ and $f_T=\big(\frac{L_T}{L_{-T}}\big)^2$, then $\mdiv{f_S}=3D_S$ and $\mdiv{f_T}=3D_T$. The Weil-pairing on $\mathcal{E}$ is defined as
\begin{gather*}
e_3\msetp{S,T}=\frac{f_S\msetp{D_T}}{f_T\msetp{D_S}} =\msetp{\frac{f_S\msetp{T}}{f_S\msetp{-T}}}^2\frac{f_T\msetp{O}}{f_T\msetp{S}}
=\msetp{\frac{L_S\msetp{T}L_O\msetp{-T}L_T\msetp{O}L_{-T}\msetp{S}}{L_O\msetp{T}L_S\msetp{-T}L_{-T}\msetp{O}L_T\msetp{S}}}^2.
\end{gather*}

Let $s\in k\msetp{S,T}\msetp{t}$ be a local coordinate at $t_0$. Choose the equations of the tangent lines such that they are also defined over $k\msetp{S,T}\msetss{s}$ and are non-zero modulo $s$. Notice that $L_O$, $L_S$, $L_T$ and $L_{-T}$ modulo $s$ are tangent lines to $E_{t_0}$ at $O$, $S$, $T$ and $-T$ respectively. Follow the construction above to obtain the Weil-pairing $e_3^{t_0}\msetp{S,T}$ on $E_{t_0}$.

Now $L_O\msetp{T}$ is a unit in $k\msetp{S,T}\msetss{s}$, because $T$ is not contained in the tangent line $L_O$ modulo~$s$ to~$E_{t_0}$ at~$O$. Similarly the other terms in the expression of $e_3\msetp{S,T}$ are units as well. Thus by construction $e_3\msetp{S,T}\!\!\mod s= e_3^{t_0}\msetp{S,T}$. Recall that $e_3\msetp{S,T}$ is a root of unity. Hence $e_3=e_3^{t_0}$.
\end{proof}

\section{The Weierstrass form}\label{hpga:sec:wf}

\begin{Proposition}\label{prop7}Let $E$ be an elliptic curve given by the Weierstrass equation $y^2z=x^3+axz^2+bz^3$ with $a,b\in k$. Then the Hesse pencil $\mathcal{E}$ can be given by
\begin{gather*}
ty^2z+3xy^2=tx^3-3ax^2z+\msetp{at-9b}xz^2+ \big(bt+a^2\big)z^3
\end{gather*}
over $k\msetp{t}$. The linear change of coordinates
\begin{gather*}
\begin{pmatrix}
x \\ y \\ z
\end{pmatrix}
=A\begin{pmatrix}
\xi \\ \eta \\ \zeta
\end{pmatrix}
\qquad\text{with}\qquad
A=\begin{pmatrix}
t & 0 & 3at^2-27bt-9a^2 \\
0 & 1 & 0 \\
-3 & 0 & t^3+9at-27b
\end{pmatrix}
\end{gather*}
transforms this into the Weierstrass form $\mathcal{E}^W\colon \eta^2\zeta=\xi^3+a_t \xi\zeta^2+b_t\zeta^3$, with
\begin{gather*}
a_t = at^4-18bt^3-18a^2t^2+54abt-\big(27a^3+243b^2\big), \\
b_t = bt^6+4a^2t^5-45abt^4+270b^2t^3+135a^2bt^2 +\big(108a^4+486ab^2\big)t-\big(243a^3b+1458b^3\big).
\end{gather*}
Moreover $\Delta\msetp{\mathcal{E}^W}=\Delta\msetp{E}\msetp{\det A}^3$ and $\det A=t^4+18at^2-108bt-27a^2$.
\end{Proposition}

Observe that the $t$ in the proposition is equal to $8t$ in the previous sections.

\begin{proof}The proof boils down to computing the map $A$, which can be found in three steps. First map the tangent line to $\mathcal{E}$ at $O$ to the line at infinity. Next scale the $z$-coordinate so that the coefficient in front of~$x^3$ and $y^2z$ are equal up to minus sign. Finally shift the $y$~-- resp.\ the $x$-coordinate so that the $xyz$, $yz^2$ resp.\ the $x^2z$ terms vanish.
\end{proof}

This proposition shows that
\begin{gather*}
j\msetp{\mathcal{E}}=1728\frac{4}{4a^3+27b^2}\msetp{\frac{at^4-18bt^3-18a^2t^2+54abt-\msetp{27a^3+243b^2}}{t^4+18at^2-108bt-27a^2}}^3.
\end{gather*}

\section{Linear change of coordinates I} \label{hpga:sec:3tlcc}

\begin{Proposition}\label{hpga:prop:pglme} Let $P_i\in\mps{2}{k}$ for $i=1,\dots,4$ be points such that no three of them are collinear. If $Q_i\in\mps{2}{k}$ for $i=1,\dots,4$ is another such set of points, then there exists a unique $A\in\mpgl{3}{k}$ such that $A\msetp{P_i}=Q_i$ for all $i=1,\dots,4$.
\end{Proposition}

This is a well-known result which is easily proved using some elementary linear algebra. Observe that an analogous result holds for two sets of $n+2$ points in $\mps{n}{k}$ such that no $n+1$ of them lie on a hyperplane.

\begin{Proposition}\label{hpga:prop:3t2pl}Let $E$ be an elliptic curve given by a Weierstrass equation defined over $k$. If $E\msets{3}=\mseta{S,T}$, then any line in $\mps{2}{\mac{k}}$ contains at most two of the following points: $O$, $S$, $T$, $S+T$.
\end{Proposition}
\begin{proof}Suppose that $L$ is a line in $\mps{2}{\mac{k}}$ containing three of the points $O$, $S$, $T$ and $S+T$. Denote these by $P_1$, $P_2$ and $P_3$. Since $E$ is given by a Weierstrass equation, $O$ is a flex point, thus $P_1+P_2+P_3=O$. However this is impossible for the points mentioned above. Hence such a line $L$ does not exist.
\end{proof}

Suppose that we are given two elliptic curves $E$ and $E^\prime$ as in the proposition above with $E\msets{3}=\mseta{S,T}$ and $E^\prime\msets{3}=\mseta {S^\prime,T^\prime}$, then Propositions~\ref{hpga:prop:pglme} and~\ref
{hpga:prop:3t2pl} imply that there exists an $A\in\mpgl{3}{\mac{k}}$ such that $O\mapsto O^\prime$, $S\mapsto S^\prime$, $T\mapsto T^\prime$ and $S+T\mapsto S^\prime+T^\prime$ and that this $A$ is unique.

\section{Linear change of coordinates II}\label{hpga:sec:wp3tiie}

\begin{Proposition}\label{hpga:prop:wprigepgl}\label{prop10} Let $E$ and $E^\prime$ be elliptic curves given by a Weierstrass equation defined over~$k$. If $\phi\colon E\msets{3}\rightarrow E^\prime\msets{3}$ is an isomorphism which respects the Weil-pairings, then there exists some $t_0\in\mps{1}{\mac {k}}$ such that the fiber $E_{t_0}$ of the Hesse pencil of $E$ admits a linear change of coordinates $\Phi\colon E_{t_0}\rightarrow E^\prime$ with $\mres{\Phi}{E\msets{3}}=\phi$.
\end{Proposition}

The essence of the proof of this proposition is the following: We determine the $t_i\in\mps{1}{\mac{k}}$ for which the $j$-invariant of $E_{t_i}$ is equal to the $j$-invariant of $E^\prime$. For each of these $t_i$'s we obtain a number of linear changes of coordinates $E_{t_i}\rightarrow E$. A counting argument shows that $\phi$ is the restriction of one of those maps. The following observation is used in the counting argument:

\begin{Lemma}\label{hpga:lem:cwpi} Let $E$ and $E^\prime$ be elliptic curves. Then $24$ out of the $48$ isomorphisms $E \msets{3}\rightarrow E^\prime\msets{3}$ respect the Weil-pairings.
\end{Lemma}
\begin{proof}Let $S,T\in E\msets{3}$ be such that $E\msets{3}=\mseta{S,T}$ and $e_3\msetp{S, T}=\zeta_3$ with $\zeta_3$ a fixed primitive third root of unity. Choose $S^\prime,T^\prime\in E^\prime\msets{3}$ likewise. Since $E\msets{3}$ and $E^\prime \msets{3}$ are two-dimensional vector spaces over $\mff_3$, there exists a~bijection
\begin{gather*}
\mgl{2}{\mff_3}\longrightarrow\miso{E\msets{3},E^\prime\msets{3}}, \\
A =\begin{pmatrix}a & b \\
c & d
\end{pmatrix}
\longmapsto \phi_A\msetp{\alpha S+\beta T} =\msetp{\alpha a+\beta b}S^\prime+\msetp{\alpha c+\beta d}T^\prime.
\end{gather*}
Notice that
\begin{gather*}
e_3^\prime\msetp{\phi_A\msetp{S},\phi_A\msetp{T}}=e_3^\prime\msetp{aS^\prime+bT^\prime,cS^\prime+dT^\prime} =e_3^\prime\msetp{S^\prime,T^\prime}^{ad-bc} =\zeta_3^{\det{A}}.
\end{gather*}
So $\phi_A$ respects the Weil-pairings if and only if $\det{A}=1$, that is $A\in \msl{2}{\mff_3}$. Now $\msetv{\mgl{2}{\mff_3}}=48$ and $\msets{\mgl{2}{\mff_3}\colon \msl{2}{\mff_3}}=2$. Hence there are 48 isomorphisms $E\msets{3}\rightarrow E^\prime\msets{3}$ of which 24 respect the Weil-pairings.
\end{proof}

Next we prove the proposition.

\begin{proof}[Proof of Proposition~\ref{hpga:prop:wprigepgl}] Let $j_0$ and $j_0^\prime$ be the $j$-invariants of $E$ and $E^\prime$ respectively. Denote the specialization of $\mathcal{E}^W$ at $t_0\in\mps{1}{\mac {k}}$ by $E_{t_0}^W$. If $E_{t_0}^W$ is non-singular, then let $A_{t_0}\colon E_{t_0} \rightarrow E_{t_0}^W$ be the isomorphism induced by the linear change of coordinates $A$ from Proposition~\ref{prop7} at $t_0$.

Assume that $j_0^\prime\neq j_0,0,1728$ and take $a_t$ as defined in Proposition~\ref{prop7}. Consider the polynomial
\begin{gather*}
G=-1728\msetp{4a_t}^3-j_0^\prime \Delta\big(\mathcal{E}^W\big) =\msetp{j_0-j_0^\prime}\Delta\msetp{E}\;t^{12}+2^{13}3^6a^2b t^{11}+\cdots.
\end{gather*}
in $k\msets{t}$, whose roots give $E_{t_0}^W$'s with $j$-invariant equal to $j_0^\prime$. The polynomial $G$ has degree~12 and its discriminant is
\begin{gather*}
-3^{147}{j_0^\prime}^8\msetp{j_0^\prime-1728}^6\Delta\msetp{E}^{44},
\end{gather*}
which is non-zero, so $G$ has distinct roots $t_1,\ldots,t_{12}$ in $\mac{k}$. Since the $j$-invariant of $E_{t_i}^W$ is equal to $j_0^\prime$, there exists an isomorphism $\Psi_i\colon E_{t_i}^W\rightarrow E^\prime$. An isomorphism respects the Weil-pairings, see \cite[Proposition~III.8.2]{silverman1985} or \cite[Theorem~3.9]{washington2008}. From Sections~\ref{hpga:sec:3tgrp} and~\ref{hpga:sec:wp} it follows that $E_{t_i} \msets{3}=E\msets{3}$ as groups with identical Weil-pairings. Therefore for every $i=1,\ldots,12$ and $\sigma\in\maut{E^\prime}\cong\mzz/2\mzz$
\begin{gather*}
\phi_{i,\sigma} =\mres{\msetp{\sigma\circ\Psi_i\circ A_{t_i}}}{E_{t_i}\msets{3}}\colon \ E\msets{3} \rightarrow E^\prime\msets{3}
\end{gather*}
is an isomorphism respecting the Weil-pairings. Notice that $\sigma\circ\Psi_i \circ A_{t_i}$ is an element of $\mpgl{3}{\mac{k}}$, because $E_{t_i}^W$ and $E^\prime$ are in Weierstrass form and $A$ is a linear change of coordinates. All 24 isomorphisms $\phi_{i,\sigma}$ are distinct as the following argument shows. Suppose that $\phi_{i,\sigma}=\phi_{j,\tau}$, then $\sigma\circ\Psi_i\circ A_{t_i}=\tau\circ\Psi_j\circ A_{t_j}$ according to Section~\ref{hpga:sec:3tlcc}. Let $P\in E^\prime\setminus E^\prime\msets{3}$, then $Q=\msetp{\sigma\circ\Psi_i \circ A_{t_i}}^{-1}\msetp{P}$ is a point in $E_{t_i}\cap E_{t_j}$, so Corollary~\ref{hpga:cor:cphpfp} implies that $t_i=t_j$, that is $i=j$. Since $\Psi_i$ and $A_{t_i}$ are isomorphisms, $\sigma=\tau$. Thus $\phi_{i,\sigma}=\phi_{j,\tau}$ if and only if $i=j$ and $\sigma=\tau$. Since the $\phi_{i,\sigma}$'s respect the Weil-pairings, Lemma~\ref{hpga:lem:cwpi} implies that these are all the possible isomorphisms $E\msets{3}\rightarrow E^\prime\msets{3}$ that respect the Weil-pairings. Hence $\phi=\phi_{i,\sigma}$ for some $i=1,\ldots,12$ and $\sigma\in\maut{E^\prime}$, which proves the proposition in this case.

Suppose that $j_0^\prime=j_0$ and $j_0^\prime\neq 0,1728$, then the $G$ above has degree 11 and the discriminant of $G$ is
\begin{gather*}
-2^{130}3^{195}a^{20}b^{10}\Delta\msetp{E}^{30},
\end{gather*}
which is again non-zero, so $G$ has distinct roots $t_1,\ldots,t_{11}$ in $\mac {k}$. In this case the $j$-invariant of~$E_\infty$ is also equal to $j_0^\prime$, so let $t_{12}=\infty$. The argument presented before now finishes the proof in this case.

Assume that $j_0^\prime=0$. This case is the same as before with the exception of the polyno\-mial~$G$, which in this case should be replaced by~$a_t$. The four distinct $t_i$'s and the six elements in $\maut{E^\prime}$ again give 24 isomorphisms $\phi_{i,\sigma}$.

Finally, if $j_0^\prime=1728$, then replace $G$ by $b_t$ (defined in Proposition~\ref{prop7}) and proceed as before.
\end{proof}

\section{Proof of the theorem}\label{hpga:sec:pot}\label{sec7}

In the proof of Theorem~\ref{hpga:thm:mt} we need a result from Galois cohomology, namely:

\begin{Lemma}\label{hpga:lem:pglgalinv}If $k$ is a perfect field, then $\mpgl{3}{\mac{k}}^{G_k}=\mpgl{3}{k}$.
\end{Lemma}
\begin{proof}Consider the short exact sequence of $G_k$-groups
\begin{gather*}
1 \longrightarrow \mac{k}^\ast \longrightarrow \mgl{3}{\mac{k}} \longrightarrow \mpgl{3}{\mac{k}} \longrightarrow 1,
\end{gather*}
which induces the exact sequence in the first row of the diagram
\begin{gather*}
\xymatrix{1\ar[r] & {\mac{k}^\ast}^{G_k}\ar[r] & \mgl{3}{\mac{k}}^{G_k}\ar[r] & \mpgl{3}{\mac{k}}^{G_k}\ar[r] & {\rm H}^1 \big(G_k,\mac{k}^\ast\big) \\
1\ar[r] & k^\ast\ar[r]\ar[u] & \mgl{3}{k}\ar[r]\ar[u] & \mpgl{3}{k}\ar[r]\ar[u] & 1.}
\end{gather*}
The second row is the definition of $\mpgl{3}{k}$ and the vertical maps are the inclusion maps. Hilbert's Theorem~90 gives that ${\rm H}^1 \big(G_k,\mac{k}^\ast\big)= \msetb{1}$. Since ${\mac{k}^\ast}^{G_k}=k^\ast$ and $\mgl{3}{\mac{k}}^{G_k}= \mgl{3}{k}$, also $\mpgl{3}{\mac{k}}^{G_k}=\mpgl{3}{k}$.
\end{proof}

\begin{proof}[Proof of Theorem~\ref{hpga:thm:mt}] Assume that $\Phi\colon E_{t_0}\rightarrow E^\prime$ for some $t_0\in\mps{1}{k}$ is an isomorphism defined over $k$. This map respects the Weil-pairings according to
\cite[Proposition III.8.2]{silverman1985}. So $\mres{\Phi}{E_{t_0}\msets{3}}\colon E_{t_0}\msets{3}\rightarrow E^\prime\msets{3}$ is a symplectic isomorphism. In Sections~\ref{hpga:sec:3tgrp} and~\ref{hpga:sec:wp} it was shown that $E\msets {3}=E_{t_0}\msets{3}$ as groups and have identical Weil-pairings. Thus $\mres {\Phi}{E_{t_0}\msets{3}}$ can be considered as a~symplectic isomorphism $E\msets {3}\rightarrow E^\prime\msets{3}$. Hence $\mres{\Phi}{E_{t_0}\msets{3}}$ is the desired map.

Suppose that there exists a symplectic isomorphism $\phi\colon E\msets{3}\rightarrow E^\prime\msets{3}$, then Proposition~\ref{hpga:prop:wprigepgl} implies that there exists a $\Phi\in\mpgl{3}{\mac{k}}$ and a $t_0\in\mps{1}{\mac{k}}$ such that $\Phi\colon E_{t_0}\rightarrow E^\prime$ and $\phi=\mres{\Phi}{E\msets{3}}$. Since $\sigma\circ\phi=\phi\circ\sigma$ for all $\sigma\in G_k$,
\begin{gather*}
\sigma\msetp{\Phi}\msetp{\sigma\msetp{S}} =\sigma\circ\Phi\msetp{S} =\sigma\circ\phi\msetp{S} =\phi\circ\sigma\msetp{S} =\Phi\msetp{\sigma\msetp{S}}
\end{gather*}
for all $S\in E\msets{3}$, so Propositions~\ref{hpga:prop:pglme} and~\ref{hpga:prop:3t2pl} imply that $\sigma\msetp{\Phi}=\Phi$. Therefore Lemma~\ref{hpga:lem:pglgalinv} implies that $\Phi\in\mpgl{3}{k}$. Hence $t_0\in\mps{1}{k}$ and $E^\prime\cong_kE_{t_0}$.
\end{proof}

\section{Characteristic two}\label{hpha:sec:char2}

So far we assumed $k$ to be a perfect field of characteristic different from two and three. There is a natural idea how to adapt the proof of Theorem~\ref{hpga:thm:mt} to characteristic two: replace the explicitly given Hesse pencil by what it actually describes, namely the pencil of cubics with the nine points of order $3$ on the initial elliptic curve as base points. This was done by one of us in her bachelor's project~\cite{tuijp2015}, and we briefly describe the results here.
\subsection{Elliptic curves in characteristic two}
Any elliptic curve $E$ over a field of characteristic $2$ can be given as follows, see \cite[p.~409]{silverman1985}:
\begin{gather*}j(E) \neq 0\colon \ y^2 +xy=x^3+a_2x^2+a_6, \qquad \Delta=a_6, \qquad j(E)=1/a_6,\\
j(E) = 0\colon \ y^2 +a_3y=x^3+a_4x+a_6, \qquad \Delta=a_3^4, \qquad j(E)=0.\end{gather*}

If $k$ is a field of characteristic 2, then the Hessian of any homogeneous polynomial $F\in k[X,Y,Z]$ of degree 3 equals zero, as is easily verified. We will show that given an elliptic curve~$E$ over~$k$, say by a special equation as above, the curves in the pencil with base points $E[3]$ all have $E[3]$ as flex points, compare Corollary~\ref{corthree} for the classical situation. In~\cite{Glynn}, Glynn defines a Hessian for any curve $C=\mathcal{Z}(F)$ with $F \in k[X,Y,Z]$ homogeneous of degree~3 (characteristic two). In fact our construction coincides with his, although we put more emphasis on how it is obtained from considering $3$-division polynomials. Note that the subject of flex points on cubic curves in characteristic two is in fact very classical: compare with, e.g., Dickson's paper~\cite{Dickson} published in~1915.

\subsection[The case $j(E) \neq 0$]{The case $\boldsymbol{j(E) \neq 0}$}\label{8.2}
We may and will assume that $E$ is given by
\begin{gather*}y^2 +xy=x^3+a_2x^2+a_6, \qquad \Delta=a_6, \qquad j=1/a_6.\end{gather*}
Define Hess($E$) as the plane curve defined by
\begin{gather*}y^2+xy^2+x^2y+xy+a_2x^3+a_2x^2+a_6x=0.\end{gather*}
The Hesse pencil $\mathcal{E}$ in this case is given by
\begin{gather*}t\big(y^2 +xy + x^3+a_2x^2+a_6\big) + y^2+xy^2+x^2y+xy+a_2x^3+a_2x^2+a_6x=0.\end{gather*}
In the next paragraphs, we will show that the Hessian and Hesse pencil have the desired pro\-per\-ties. Firstly, the analog of Proposition~\ref{two} holds:

\begin{Proposition}\label{flexE} If $P$ is a point on an elliptic curve $E$ with equation $y^2+xy=x^3+a_2x^2+a_6$, then $P$ is a flex point of $E$ if and only if $P \in E \cap \operatorname{Hess}(E)$.
\end{Proposition}
\begin{proof}The point $O$ is a flex point on both $E$ and $\operatorname{Hess}(E)$ hence the result holds for $O$. Next take any other flex point of~$E$, i.e., a point $P\neq O$ with $3P=O$. Put $P=(x,y)$. Note that $x\neq 0$, because any point $(0,y)\in E$ has order two. A small calculations (compare $x$-coordinates of $-P$ and $2P$) shows~$P$ is a flex point precisely when
\begin{gather*}
 (E1)\quad 0=x^2 + \left(\frac{y}{x}\right)^2 + \frac{y}{x} + a_2, \\
 (E2)\quad y+x = \left(x + \frac{y}{x} + 1\right)\left(x^2 + \left(\frac{y}{x}\right)^2 + x + \frac{y}{x} + a_2\right) + x^2.
\end{gather*}
Using the equation defining $E$, one rewrites $(E1)$ as
\begin{gather} \label{eq2}
 0 = y^2+xy^2+x^2y+xy+a_2x^3+a_2x^2+a_6x.
\end{gather}
The proposition now follows from a straightforward calculation.
\end{proof}

Note that in fact $P=(x,y)$ is a flex point on $E$ if and only if $P\in E$ satisfies equation~(\ref{eq2}).

Now we show the analog of Corollary~\ref{corthree}.

\begin{Proposition}If $P$ is a flex point on an elliptic curve $E$ given by $y^2+xy=x^3+a_2x^2+a_6$, then $P$ is also a flex point on the Hesse pencil $\mathcal{E}$.
\end{Proposition}
\begin{proof}It follows directly from the construction of $\mathcal{E}$ that $P$ is indeed a point on it. To prove that $P$ is also a flex point on $\mathcal{E}$, one shows that the tangent line to $\mathcal{E}$ at $P$ intersects $\mathcal{E}$ at $P$ with multiplicity 3. This is a straightforward calculation for which we refer to~\cite{tuijp2015}. Clearly the point $O$ is also a point on the Hesse pencil and it is also a flex point, as can be shown in the same way.
\end{proof}
\subsection[The case $j(E) =0$]{The case $\boldsymbol{j(E) =0}$}\label{8.3}
In the remaining case $j(E)=0$ we may and will assume that $E$ is given as
\begin{gather*}y^2 +a_3y=x^3+a_4x+a_6, \qquad \Delta=a_3^4, \qquad j=0.\end{gather*}
Now define Hess($E$) by
\begin{gather*}xy^2 + a_3 xy+ a_4 x^2 + \big(a_3^2 + a_6\big)x + a_4^2=0,\end{gather*}
so the Hesse pencil $\mathcal{E}$ becomes
\begin{gather*}t\big(y^2 +a_3y + x^3+a_4x+a_6\big) + xy^2 + a_3 xy+ a_4 x^2 + \big(a_3^2 + a_6\big)x + a_4^2 =0.\end{gather*}
In this case as well, the analogs of Proposition~\ref{two} and Corollary~\ref{corthree} hold:

\begin{Proposition}If $P$ is a point on the elliptic curve $E$ given by $y^2+a_3y=x^3+a_4x+a_6$, then $P$ is a flex point if and only if $P \in E \cap \operatorname{Hess}(E)$.
\end{Proposition}

\begin{proof}For $O$ the exact same argument holds as when $j(E)\neq 0$. A calculation shows that $P=(x,y)\in E$ is a flex point precisely when
\begin{gather*}
a_3^2x = x^4 + a_4^2.\end{gather*}
Using the equation of the elliptic curve this is rewritten as
\begin{gather*}
0= a_3^2x + x\big(y^2 + a_3y + a_4x + a_6\big) + a_4^2 = xy^2 + \big(a_3^2+a_6\big)x + a_4x^2 + a_3xy + a_4^2.\tag*{\qed}
\end{gather*}\renewcommand{\qed}{}
\end{proof}

\begin{Proposition} If $P$ is a flex point on the elliptic curve $E$ given by $y^2+a_3y=x^3+a_4x+a_6$, then it is also a flex point on the Hesse pencil $\mathcal{E}$.
\end{Proposition}

\begin{proof}The reasoning is the same as for the case $j(E) \neq 0$. The straightforward calculation is presented in detail in~\cite{tuijp2015}.
\end{proof}

Using the properties shown above of our Hesse pencil in characteristic two, we can now almost completely follow the reasoning of the earlier sections since most arguments do not involve the characteristic of~$k$. Only for the analog of Proposition~\ref{prop10} the proof needs to be adjusted in characteristic two, because here actual calculations are done with the Hesse pencil. We state it in the present situation.

\begin{Proposition}\label{4.10}Let $E$ and $E'$ be elliptic curves given by a Weierstrass equation defined over~$k$. If $\phi \colon E[3] \rightarrow E'[3]$ is an isomorphism which respects the Weil-pairings, then there exists a linear change of coordinates $\Phi\colon E_{t_0} \rightarrow E'$ for some $t_0 \in \mathbb{P}^1(\bar{k})$ such that $\Phi|_{E[3]} = \phi$.
\end{Proposition}

The remainder of this section consists of proving Proposition~\ref{4.10}.

\subsection[The case $j(E) \neq 0$]{The case $\boldsymbol{j(E) \neq 0}$}

We first determine the Weierstrass form of the Hesse pencil in the present case.

Rewrite its equation (as introduced in Section~\ref{8.2}) as
\begin{gather*} (t+1)y^2z+(t+1)xyz+xy^2+x^2y +(t+a_2)x^3+a_2(t+1)x^2z+a_6xz^2+ta_6z^3 =0.\end{gather*}
By a suitable change of coordinates this can be brought in Weierstrass form $\eta^2\zeta + b_1\xi \eta \zeta = \xi^3 + b_2 \xi^2 \zeta + b_6 \zeta^3$. This is quite analogous to what was sketched in the proof of Proposition~\ref{prop7}; see~\cite{tuijp2015} for a detailed description and~\cite{ART} for a more general setup. The stated equation is found with
\begin{gather*}
b_1 = (t+1)^2,\qquad
b_2 = a_2(t+1)^4+a_6t(t+1),\qquad
b_6 = {a_6}(t^4+t^3+t^2+t+a_6)^3.
\end{gather*}

Denote this family of curves by $\mathcal{E}^W$ and let an individual curve in the pencil be denoted by~$E_t^W$, then
\begin{gather*}
j\big(E^W_t\big)=b_1^6/b_6 = \frac{(t+1)^{12}}{a_6\big(t^4+t^3+t^2+t+a_6\big)^3}.
\end{gather*}
If $t=1$, the transformations needed to obtain the above Weierstrass form do not work. In this case one transforms the fiber of given pencil, so the curve $E_1$ with equation
\begin{gather*}xy^2+x^2y+(1+a_2)x^3+a_6xz^2+a_6z^3=0,\end{gather*}
into the other Weierstrass form in characteristic 2:
\begin{gather*}\eta^2\zeta + b_3 \eta \zeta^2 = \xi^3 + b_4 \xi \zeta^2 + b_6 \zeta^3.\end{gather*}
Explicitly, this results in the equation
\begin{gather*}
 \eta^2 \zeta + a_6 \eta \zeta^2 + \xi^3 + a_6^2 \xi \zeta^2 + a_6^2(1+a_2) \zeta^3 =0.
\end{gather*}
In this way one obtains for every $t$ a projective, linear transformation $E_t \rightarrow E_t^W$. Let us denote this transformation by $A_t$.

\begin{proof}[Proof of Proposition~\ref{4.10} for the case $\boldsymbol{j(E) \neq 0}$]Given another elliptic curve $E'$ with $j$-in\-variant $j_0'$, we want to determine $t$ for which our Hesse pencil has the same $j$-invariant. First, let us assume that $j_0'$ is nonzero and not equal to $j_0$. Then
\begin{gather*}
j_0'=j\big(E^W_t\big) \ \Leftrightarrow \ (t+1)^{12} = j_0'a_6\big(t^4+t^3+t^2+t+a_6\big)^3.
\end{gather*}
Define the polynomial
\begin{gather*}G=(t+1)^{12} + j_0'a_6\big(t^4+t^3+t^2+t+a_6\big)^3.\end{gather*}
The zeros of this polynomial are precisely all $t_0$ such that $j\big(E_{t_0}^W\big)=j_0'$. The discriminant of $G$ equals $a_6^{44} j_0'^{14}$, which is nonzero, because $j_0'$ and $a_6$ are nonzero. We conclude that precisely $12$ values $t_0\in\bar{k}$ exist which give the desired $j$-invariant.

For every $t_0$, there is an isomorphism $A_{t_0}$ between $E_{t_0}$ and $E_{t_0}^W$, induced by the change of coordinates seen above. For every $t_0$ which is moreover a zero of $G$, there is an isomor\-phism~$\Psi _{t_0}$ between $E_{t_0}^W$ and $E'$, because these curves have equal $j$-invariants. Lastly, there exist 2 automorphisms $\sigma$ of $E'$ \cite[p.~410]{silverman1985}. Taking the composition of these three isomorphisms and restricting to the 3-torsion group $E_{t_0}[3]$, which equals $E[3]$, we obtain $12 \times 2 = 24$ isomor\-phisms~$\phi_{t_0,\sigma}$; they are described as
\begin{gather*}\phi_{t_0,\sigma} = \sigma \circ \Psi _{t_0} \circ A_{t_0}| _{E_{t_0}[3]}\colon \ E[3] \rightarrow E'[3].\end{gather*}
These 24 isomorphisms are pairwise distinct and respect the Weil-pairing (see Section~\ref{hpga:sec:wp3tiie}, observe that this argument is independent of the characteristic of $k$).

Now consider the case $j_0'=j_0 \neq 0$. Then $j_0'a_6=1$ since $j_0=1/a_6$. Our polynomial $G$ therefore has degree~11 and discriminant~$a_6^{30} \neq 0$. So this gives us 11 pairwise distinct $t\in\bar{k}$ such that $j\big(E_t^W\big)=j_0$. Another curve with this $j$-invariant is $E_{\infty}=E$. So again we find 12 distinct $t$'s and in the same way as above, we find 24 isomorphisms respecting the Weil-pairing.

If $j_0'=0$, the only $t$-value with $j(E_t)=0$ is $t=1$. Because $E'$ has $j$-invariant zero and $k$ has characteristic 2, its automorphism group has 24 elements \cite[p.~410]{silverman1985}. So again we find~24 isomorphisms respecting the Weil-pairing.

We now complete the proof of Proposition~\ref{4.10} for the case $j(E)\neq 0$ by the exact same argument as presented in the proof of Proposition~\ref{prop10}.
\end{proof}

\subsection[The case $j(E) = 0$]{The case $\boldsymbol{j(E) = 0}$}

For $j(E)=0$ the calculations are slightly more involved. Bringing the Hesse pencil
\begin{gather*}
t\big(y^2z + a_3yz^2 + x^3 + a_4xz^2 + a_6z^3\big) = xy^2 + a_3 xyz+ a_4 x^2z + \big(a_3^2 + a_6\big)xz^2 + a_4^2z^3
\end{gather*}
in Weierstrass form, one obtains $\mathcal{E}^W$ of the form
\begin{gather*}
\eta^2\zeta + \xi \eta \zeta + \xi^3 +b_2 \xi^2 \zeta + b_6 \zeta^3 = 0
\end{gather*}
with $b_2$, $b_6$ explicit rational expressions in the $a_j$. The $j$-invariant of $E_t^W$ for $t \neq0$ is
\begin{gather*}
j\big(E^W_t\big) = \frac{{a_3}^8}{\big(t^4+a_3^2t+a_4^2\big)^3}.
\end{gather*}
If $t=0$, so if $E_t=E_0$ is the Hessian curve, the transformations needed here are not valid. Therefore we treat this case separately. The Hessian here is given by
\begin{gather*}xy^2 + a_3xyz + a_4x^2z + \big(a_3^2 + a_6\big)xz^2 + a_4^2z^3=0.\end{gather*}
This results in the Weierstrass equation
\begin{gather*}\eta^2\zeta + a_3^2\xi \eta\zeta + \xi^3+ \big({a_3^4 + a_6}{a_3^2}\big)\xi^2\zeta + {a_3^4a_4^6 }\zeta^3=0\end{gather*}
with $j$-invariant $\frac{a_3^8}{a_4^6}$. We conclude that for every $t$ including $t=0$, the $j$-invariant of $E_t^W$ is given by
\begin{gather*}
j\big(E_t^W\big) = \frac{{a_3}^8}{\big(t^4+a_3^2t+a_4^2\big)^3}.
\end{gather*}
Moreover, again we have a projective linear automorphism $A_t\colon E_t\to E_t^W$.

\begin{proof}[Proof of Proposition~\ref{4.10} if $\boldsymbol{j(E) = 0}$] Again, we want to show that for every $j_0'$, there exist 24 different isomorphisms. First, assume that $j_0' \neq 0$. Define
\begin{gather*} G:=a_3^8 + j_0'\big(t^4+a_3^2t+a_4^2\big)^3 ,
\end{gather*}
which has degree 12 and discriminant $a_3^{176}j_0'^{14} = \Delta (E)^{44} j_0'^{14}$. Therefore $G$ has $12$ pairwise distinct zeros, which are all solutions $t_0$ such that $j\big(E^W_{t_0}\big)=j_0'$. Again, together with the 2 automorphisms $\sigma$, we find 24 isomorphisms.

Now assume that $j_0'=0$. In this case, the curve in the Hesse pencil we are looking for, is~$E$ itself: this curve has $j$-invariant zero. And again the automorphism group has order 24 so also in this case, we find 24 isomorphisms again, and the proof of Proposition~\ref{4.10} in this case is completed using the same reasoning as before (compare the proof of Proposition~\ref{prop10}).
\end{proof}

Using Proposition~\ref{4.10} one concludes that Theorem~\ref{hpga:thm:mt} holds in characteristic two as well, using the reasoning as presented in Section~\ref{sec7}.

\section{Comparison with the literature}\label{hpga:sec:cwer}

Theorem~\ref{hpga:thm:mt} is part of a more general problem: Given an elliptic curve $E$ over a field $k$ and an integer $n$, describe the universal family of elliptic curves $\mathcal{E}$ such that for each member~$\mathcal{E}_{t_0}$ the Galois representations on $E\msets{n}$ and $\mathcal{E}_{t_0}\msets{n}$ are isomorphic and the isomorphism is symplectic. For various $n$ explicit families are known in the literature.

In~\cite{rubin1993} Rubin and Silverberg construct for any elliptic curve over $\mqq$ such an explicit family for $n=3$ and $n=5$. Their proofs are motivated by the theory of modular curves. Our Theorem~\ref{hpga:thm:mt} corresponds roughly to~\cite[Theorem~4.1]{rubin1993} and~\cite[Remark~4.2]{rubin1993}.

Using invariant theory and a generalization of the classical Hesse pencil, Fisher in~\cite{fisher2012} describes such families for elliptic curves defined over a perfect field of characteristic not dividing~$6n$ with $n=2,3,4,5$. Theorem~\ref{hpga:thm:mt} is a special case of~\cite[Theorem~13.2]{fisher2012}. It is unclear whether Fisher's proof of \cite[Theorem~13.2]{fisher2012} can be adapted to the case of characteristic two. In~\cite{Fisher14} Fisher moreover treats the cases $n=7$ and $n=11$.

The Hesse pencil is used by Kuwata in~\cite{kuwata2012}. For any elliptic curve $E$ over a number field he constructs two families of elliptic curves such that for each member the Galois representation on its 3-torsion is equivalent to the one on $E\msets{3}$. In the first family the isomorphism of the 3-torsion groups is symplectic, whereas in the second family the isomorphism is anti-symplectic. The proofs use classical projective geometry and the classification of rational elliptic surfaces. Theorem~\ref{hpga:thm:mt} is essentially \cite[Theorem~4.2]{kuwata2012} (although our proof is more detailed and totally elementary, and moreover we extend the result to characteristic two). Notice that the Weierstrass form of the Hesse pencil in \cite[Remark~4.4]{kuwata2012} is the same as the one in Proposition~\ref{prop7} with $t$ replaced by $t^{-1}$ and the $x$ and~$y$ coordinates scaled by some power of~$t$.

An overview of results on the classical Hesse pencil is given by Artebani and Dolgachev in~\cite{artebani2009}.

\subsection*{Acknowledgements}
It is a pleasure to thank the referees of an earlier draft of this text. Their comments and suggestions were greatly appreciated. We also thank Matthias Sch\"{u}tt for helpful remarks.

\pdfbookmark[1]{References}{ref}
\LastPageEnding

\end{document}